\documentclass{amsart}
\usepackage{amsmath}
\usepackage{url}

\theoremstyle{definition}

\theoremstyle{remark}

\numberwithin{equation}{section}

\reversemarginpar

\newcommand{\no}{\noindent}

\newcommand{\realpart}{\mathop{\rm Re}\nolimits}

\newcommand{\ba}{\begin{eqnarray}}
\newcommand{\ea}{\end{eqnarray}}

\newtheorem{Definition}{\bf Definition}[section]
\newtheorem{Thm}[Definition]{\bf Theorem}
\newtheorem{Lem}[Definition]{\bf Lemma}
\newtheorem{Note}[Definition]{\bf Note}
\newtheorem{Prop}[Definition]{\bf Proposition}
\newtheorem{Cor}[Definition]{\bf Corollary}

\begin{document}

\title {The evaluation of a quartic integral via Schwinger, Schur and Bessel}

\author{Tewodros Amdeberhan}
\address{Department of Mathematics,
Tulane University, New Orleans, LA 70118}
\email{tamdeber@tulane.edu}

\author{Victor H. Moll}
\address{Department of Mathematics,
Tulane University, New Orleans, LA 70118}
\email{vhm@math.tulane.edu}

\author{Christophe Vignat}
\address{Laboratoire des Signaux et Systemes, Universite d'Orsay, France}
\email{christophe.vignat@u-psud.fr}


\date{\today}


\begin{abstract}
 We provide additional methods for the evaluation of the integral 
\begin{eqnarray}
N_{0,4}(a;m) & := & \int_{0}^{\infty} \frac{dx}
{\left( x^{4} + 2ax^{2} + 1 \right)^{m+1}} \nonumber
\end{eqnarray}
\noindent
where $m \in {\mathbb{N}}$ and $a \in (-1, \infty)$ in the form
\begin{eqnarray}
N_{0,4}(a;m) & = & \frac{\pi}{2^{m+3/2} (a+1)^{m+1/2} } 
P_{m}(a)  \nonumber
\end{eqnarray}
\noindent
where $P_{m}(a)$ is a polynomial in $a$. The first one is based on a method 
of Schwinger to evaluate integrals appearing in Feynman diagrams, the 
second one is a byproduct of an expression for a rational integral in 
terms of Schur functions. Finally, the third proof, is obtained from an 
integral representation involving modified Bessel functions. 

\end{abstract}

\maketitle


\vskip 20pt 

\section{Introduction} \label{S:intro} 

The definite integral 
\begin{equation}
N_{0,4}(a;m) = \int_{0}^{\infty} \frac{dx}{(x^{4}+2ax^{2}+1)^{m+1}},
\end{equation}
\noindent 
with $m \in \mathbb{N}$ and $a>-1$ has the value 
\begin{equation}
N_{0,4}(a;m)  =   \frac{\pi}{2^{m+3/2} (a+1)^{m+1/2} } 
P_{m}(a)  \label{identity-1}
\end{equation}
\noindent 
where 
\begin{equation}
P_{m}(a) = 2^{-2m} \sum_{k=0}^{m} 2^{k} \binom{2m-2k}{m-k} \binom{m+k}{k} 
(a+1)^{k}.
\label{polyP-def}
\end{equation}

The reader will find in \cite{amram} a survey of the many proofs of 
(\ref{identity-1}) that have appeared in the literature. After this survey 
was written, other proofs have appeared; see  
\cite{apagodu10,berg-vignat, gonzalez3,koutschan1}.  The goal of this note 
is to present some additional proofs: the first one is based in 
the {\em Schwinger  parametrization}, a second one using {\em Schur 
functions} and the last one involves a representation involving  integrals
of {\em Bessel functions}.

\section{Schwinger parametrization} \label{S:Fey} 

The method of Schwinger  parameters,  described in classical quantum
theory books such as \cite{peskin}, is utilized here to give a proof of 
(\ref{identity-1}). A preliminary discussion is presented first.

Starting from \begin{equation}
\frac{1}{A_{1}A_{2}}= 
\int_{0}^{1}\frac{dx}{\left(xA_{1}+\left(1-x\right)A_{2}\right)^{2}}
=\int_{0}^{1} \int_{0}^{1}\frac{\delta\left(1-x_{1}-x_{2}\right)}
{\left(x_{1}A_{1}+x_{2}A_{2}\right)^{2}}dx_{1}dx_{2}
\label{eq:A1A2}\end{equation}
\noindent
where $\delta$ is the Dirac measure, it 
can be deduced by induction that 
\begin{equation}
\frac{1}{\prod_{i=1}^{n}A_{i}}=
\left(n-1\right)!\int_{[0,1]^{n}}
\frac{\delta\left(1-\sum_{i=1}^{n}x_{i}\right)}
{\left(\sum_{i=1}^{n}x_{i}A_{i}\right)^{n}}\prod_{i=1}^{n}dx_{i}.
\label{eq:A1A2An}\end{equation}
The induction step uses the $n$-times differentiated version of \eqref{eq:A1A2},
namely
\begin{equation}
\frac{1}{A_{1}A_{2}^{\nu_{2}}}=
\int_{0}^{1}\int_{0}^{1}\frac{\delta\left(1-x_{1}-x_{2}\right)
\nu_{2}x_{2}^{\nu_{2}-1}}{\left(x_{1}A_{1}+x_{2}A_{2}\right)^{\nu_{2}+1}}
dx_{1}dx_{2}.
\end{equation}
Repeated differentiations of \eqref{eq:A1A2An} yields the more general
result\begin{equation}
\frac{1}{\prod_{i=1}^{n}A_{i}^{\nu_{i}}}=
\frac{\Gamma\left(\nu\right)}{\prod_{i=1}^{n}\Gamma\left(\nu_{i}\right)}
\int_{[0,1]^{n}}
\frac{\delta\left(1-\sum_{i=1}^{n}x_{i}\right)
\prod_{i=1}^{n}x_{i}^{\nu_{i}-1} \, dx_{i}}
{\left(\sum_{i=1}^{n}x_{i}A_{i}\right)^{\nu}}
\label{eq:general formula}
\end{equation}
with $\nu= \nu_{1} + \cdots + \nu_{n}$ and $A_{i} \in \mathbb{C}$.

An alternative proof has been provided by Shapiro 
in his lecture notes \cite{shapiro}. This direct proof of \eqref{eq:A1A2An}
starts with
\begin{equation}
\frac{1}{A}=\int_{0}^{\infty}e^{-Ax}dx
\end{equation}
for $\realpart{A}>0$.  It follows that

\begin{equation}
\prod_{i=1}^{n}\frac{1}{A_{i}}=
\int_{\mathbb{R}_{+}^{n}}e^{-\sum_{i=1}^{n}A_{i}x_{i}}\prod_{i=1}^{n}dx_{i}
\end{equation}
Now let $x=x_{1}+ \cdots x_{n}$ and $y_{i}=\frac{x_{i}}{x}$ to obtain
\begin{equation}
\prod_{i=1}^{n} dx_{i} 
=x^{n-1}dx\prod_{i=1}^{n}\delta\left(1-\sum_{i=1}^{n}y_{i}\right) 
\, dy_{i}
\end{equation}
\noindent
so that
\begin{equation}
\prod_{i=1}^{n}\frac{1}{A_{i}}=
\int_{\mathbb{R}_{+}^{n}}\prod_{i=1}^{n}
\delta\left(1-\sum_{i=1}^{n}y_{i}\right)x^{n-1}dx\, 
e^{-x\sum_{i=1}^{n}A_{i}y_{i}} \, dy_{i}.
\end{equation}
After a change of variable, the latter integral is evaluated as
\begin{equation}
\int_{0}^{\infty} x^{n-1} e^{-x \sum_{i=1}^{n} A_{i}y_{i}} dx = 
\Gamma(n) \left( \sum_{i=1}^{n} A_{i}y_{i} \right)^{-n}
\end{equation}
\noindent
resulting in 
\begin{equation}
\prod_{i=1}^{n}\frac{1}{A_{i}}
=\Gamma\left(n\right) \int_{\mathbb{R}_{+}^{n}}\prod_{i=1}^{n}
\frac{\delta\left(1-\sum_{i=1}^{n}y_{i}\right)}
{\left(\sum_{i=1}^{n}A_{i}y_{i}\right)^{n}} \, dy_{i}.
\end{equation}

Shapiro \cite{shapiro} states: 
\begin{quotation}
An interesting anecdote of physics history is that Schwinger remained
bitter that a virtually identical mathematical trick became commonly
known as Feynman parameters. Why two brilliant physicists, each of
whom had an appropriately won a Nobel prize, should fight over what
is essentially a trivial mathematical trick, is an interesting question
in the sociology of physicists.
\end{quotation}\
An additional remark is that this representation is also known in
probability theory, and was apparently introduced by Mauldon \cite{mauldon}.
The $n$-dimensional Dirichlet distribution 
with parameters $\left\{ \nu_{i}\right\} _{1\le i\le n}$
of a random vector $X$ reads\[
f_{X}\left(x_{1},\dots,x_{n}\right)=\frac{\Gamma\left(\nu\right)}{\prod_{i=1}^{n}\Gamma\left(\nu_{i}\right)}\prod_{i=1}^{n}x_{i}^{\nu_{i}-1}\delta\left(1-\sum_{i=1}^{n}x_{i}\right)\]
with $\nu = \nu_{1} + \ldots + \nu_{n}$. Mauldon  
proved that a Dirichlet distributed random vector $X$ satisfies the following 
equality
\begin{equation}
E_{X}\left(\sum_{i=1}^{n}\lambda_{i}X_{i}\right)^{-\nu}
=\sum_{i=1}^{n}\lambda_{i}^{-\nu_{i}},
\end{equation}
which is exactly \eqref{eq:general formula}.

\medskip

This procedure is now applied to the integral $N_{0,4}(a;m)$ written as 
\begin{equation}
N_{0,4}(a;m) = \int_{0}^{\infty} 
\frac{dx}{(1+ e^{i \theta} x^{2})^{m+1} \, \, 
(1+ e^{-i \theta} x^{2})^{m+1}} \, dx 
\end{equation}
\noindent
with $a = \cos \theta$ and $\theta \in (-\pi,\pi)$. The corresponding 
parameters 
are $n=2$ and $\nu_{1}=\nu_{2} = m+1$. 

Recall the expression for the beta function 
\begin{equation}
B(a,b)  = \int_{0}^{1} x^{a-1}(1-x)^{b-1} \, dx
\end{equation}
\noindent
and its relation to the gamma function
\begin{equation}
B(a,b) = \frac{\Gamma(a+b)}{\Gamma(a) \, \Gamma(b)}
\end{equation}
\noindent
to write
\begin{equation}
N_{0,4}(a;m) = \frac{1}{B(m+1,m+1)} 
\int_{0}^{\infty} \, dx \int_{0}^{1} 
\frac{x_{1}^{m} (1-x_{1})^{m} \, dx_{1}} 
{ \left[ 1 + x^{2} \times  (x_{1}e^{i \theta} + 
(1-x_{1})e^{-i \theta} )\right]^{2m+2}}.
\nonumber
\end{equation}
\noindent
Formula $3.194.3$ in \cite{gr} gives 
\begin{equation}
\int_{0}^{\infty} \frac{dx}{(1+ux^{2})^{\alpha}} = 
\frac{1}{2 \sqrt{u}} B \left( \tfrac{1}{2}, \alpha - \tfrac{1}{2} \right)
\end{equation}
\noindent 
valid for  $u \in \mathbb{C}$ with 
$|\text{Arg}(u) | < \pi$. This yields 
\begin{equation}
N_{0,4}(a;m) = \frac{B(\tfrac{1}{2},2m+\tfrac{3}{2})}{2 B(m+1,m+1)} 
\int_{0}^{1} \frac{x_{1}^{m} (1-x_{1})^{m} \, dx_{1}} 
{ \sqrt{x_{1}e^{i \theta} + (1-x_{1})e^{-i \theta}}}.
\label{formula-10}
\end{equation}

The next lemma provides an evaluation of the integral in (\ref{formula-10}).

\begin{Lem}
Define 
\begin{equation}
I_{m}(z) := \int_{0}^{1} \frac{x^{m}(1-x)^{m} \, dx}{\sqrt{1 + (z^{2}-1)x}}.
\end{equation}
\noindent
Then 
\begin{equation}
I_{m}(z) = \frac{2^{2m+2}}{(m+1) \binom{2m+1}{m} \binom{4m+2}{2m+1} 
(z+1)^{2m+1}} \sum_{k=0}^{m} \binom{2m-2k}{m-k} \binom{m+k}{k} (z+1)^{2k} 
z^{m-k}.
\nonumber
\end{equation}
\end{Lem}
\begin{proof}
Expand the integrand in powers of $x$ and integrate. The result is simplified 
using the value for the beta function 
\begin{equation}
B(a,b) = \frac{a+b}{ab} \binom{a+b}{a}^{-1}
\end{equation}
\noindent
for $a, \, b \in \mathbb{N}$. It follows that 
\begin{eqnarray}
I_{m}(z) & = & \sum_{k=0}^{\infty} \binom{-\tfrac{1}{2}}{k} (z^{2}-1)^{k} 
\int_{0}^{1} x^{m+k} (1-x)^{m} \, dx \nonumber \\
& = & \sum_{k=0}^{\infty} \binom{-\tfrac{1}{2}}{k} B(m+k+1,m+1) (z^{2}-1)^{k}
\nonumber \\
& = & \frac{1}{m+1} \sum_{k=0}^{\infty} \binom{2k}{k} \binom{2m+k+1}{m+1}^{-1} 
\left( \frac{1-z^{2}}{4} \right)^{k}.
\nonumber 
\end{eqnarray}
\noindent
The result now follows from the identity
\begin{multline}
 \sum_{k=0}^{\infty} \binom{2k}{k} \binom{2m+k+1}{m+1}^{-1} 
\left( \frac{1-z^{2}}{4} \right)^{k} =  \label{formula-00} \\
\frac{2^{2m+2}}{\binom{2m+1}{m} \binom{4m+2}{2m+1} 
(z+1)^{2m+1}} \sum_{k=0}^{m} \binom{2m-2k}{m-k} \binom{m+k}{k} (z+1)^{2k} 
z^{m-k}.
\end{multline}
\indent
To establish (\ref{formula-00}), the Wilf-Zeilberger automated method 
\cite{aequalsb} is employed. It provides the recurrence 
\begin{multline}
(z^{2}-1)^{2} (4m+9)(4m+7)a_{m+2} - (z^{4}-6z^{2}+1)(2m+6)(2m+3)a_{m+1} - 
\\
4z^{2}(m+2)(m+3)a_{m} = 0
\nonumber
\end{multline}
\noindent
that both sides of (\ref{formula-00}) satisfy. Checking the initial values 
at $m=0$ and $m=1$ is elementary and the proof is complete.
\end{proof}

\smallskip

The evaluation of (\ref{identity-1}) is restated in a slightly different 
form, as the content of the next statement.

\begin{Thm}
Let $m \in \mathbb{N}, \, z = e^{i \theta}$ and $a  = \cos \theta$. Then
\begin{equation}
N_{0,4}(a;m) = \frac{B(\tfrac{1}{2}, 2m+ \tfrac{3}{2})}{2 B(m+1,m+1)} 
\sqrt{z}I_{m}(z)
\end{equation}
\noindent
holds. 
\end{Thm}

\section{A connection with Schur functions} \label{S:schur} 

The integral
\begin{equation}
G_{n}(\mathbf{q}) = \frac{2}{\pi} \int_{0}^{\infty} \prod_{k=1}^{n} 
\frac{1}{x^{2}+q_{k}^{2}} \,dx,
\end{equation}
\noindent
is clearly a symmetric function of the parameters $q_{1}, \ldots, q_{n}$. It has
been expressed in \cite{aems-1} as
\begin{equation}
G_{n}(\mathbf{q}) = \frac{s_{\lambda(n-1)}(\mathbf{q})}
                {e_{n}(\mathbf{q}) \, s_{\lambda(n)}(\mathbf{q})},
\label{schur-1}
\end{equation}
\noindent
where $\lambda(n)$ is the partition $\lambda(n) = (n-1,n-2,\ldots,1)$, 
$e_{n}(\mathbf{q}) = q_{1}q_{2}\cdots q_{n}$ and the Schur function
corresponding to a partition $\mu$ defined by 
\begin{equation}
s_{\mu}(\mathbf{q}) = \frac{a_{\mu+\lambda(n)}(\mathbf{q})}
                           {a_{\lambda(n)}(\mathbf{q})}
\end{equation}
\noindent
with $a_{\mu}(\mathbf{q}) := \text{det}(q_{i}^{\mu_{j}})_{1\leq i,j 
\leq n}$. 

\smallskip

In the current problem, let $a = \tfrac{1}{2}(w^{2}+w^{-2})$ so that
\begin{equation}
N_{0,4}(a;m) =  \int_{0}^{\infty} \frac{dx}{(x^{2}+w^{2})^{m+1} \, 
(x^{2}+w^{-2})^{m+1}}.
\end{equation}
\noindent
Now take the partition $\lambda$ as above 
and the $(2m+2)-$tuple $\mathbf{q}$ given by 
\begin{equation}
\lambda(n) = (n-1, \, n-2, \ldots, 1), \quad \mathbf{q}
 = (w, \ldots, w,w^{-1},\ldots,
w^{-1})
\end{equation}
\noindent
where $\mathbf{q}$ has $m+1$ copies of $w$ and 
$w^{-1}$. The expression (\ref{schur-1}), that appears as 
Theorem 5.1 of \cite{aems-1},
provides the next result. 

\begin{Thm}
\label{thm-schur}
The quartic integral is given by 
\begin{equation}
N_{0,4}(a;m) = \frac{\pi}{2} \frac{s_{\lambda(2m+1)}(\mathbf{q})}
{s_{\lambda(2m+2)}(\mathbf{q})}. 
\end{equation}
\end{Thm}

As a corollary, this implies a specialization for the Schur functions. 

\begin{Cor}
If $q = (w, \ldots,w,w^{-1},\ldots,w^{-1})$ is of length $2m+2$, then
\begin{equation}
\frac{
s_{\lambda(2m+1)}(w,\ldots,w,w^{-1},\ldots,w^{-1}) }
{s_{\lambda(2m+3)}(w,\ldots,w,w^{-1},\ldots,w^{-1}) }
=
2^{-2m} \sum_{k=0}^{m} \binom{2k}{k} \binom{2m-k}{m} 
\frac{1}{(w+w^{-1})^{2k+1}}.
\nonumber
\end{equation}
\end{Cor}

Theorem \ref{thm-schur} has a natural generalization to integrals of the form 
\begin{equation}
\int_{0}^{\infty} \frac{dx}{Q_{2n}(x)^{m}},
\nonumber
\end{equation}
\noindent
where 
\begin{equation}
Q_{2n}(x) = x^{2n} + a_{1}x^{2n-2} + a_{2}x^{2n-4} + \cdots + a_{2}x^{4} + 
a_{1}x^{2} + 1
\end{equation}
\noindent
is a {\em palindromic polynomial}. These polynomials factor as 
\begin{equation}
Q_{2n}(x) = (x^{2}+w_{1}^{2})(x_{2}+w_{1}^{-2}) \cdots 
(x^{2}+w_{n}^{2})(x_{2}+w_{n}^{-2}).
\end{equation}

Consider the partition $\lambda$ and the $n$-concatenation 
${\mathbf{q}} = (q_{1}, \ldots, q_{n})$, of the $2m$-tuples $q_{j}$, by 
\begin{equation}
\lambda(s) = (s-1,s-2,\ldots,1), \text{   and   }
q_{j} = (w_{j},\ldots,w_{j};w_{j}^{-1},\ldots,w_{j}^{-1})
\end{equation}
\noindent
where $q_{j}$ has $m$ copies of each $w_{j}$ and $w_{j}^{-1}$. A direct 
application of (\ref{schur-1})  gives the next result.

\begin{Prop}
Preserving the notation from above,
\begin{equation}
\int_{0}^{\infty} \frac{dx}{Q_{2n}(x)^{m}} = 
\frac{\pi}{2} \frac{s_{\lambda(2nm-1)}(q)}{s_{\lambda(2nm)(q)}}.
\end{equation}
\end{Prop}

\section{A connection with Bessel functions} \label{S:Bessel} 

The quartic integral (\ref{identity-1}) 
is evaluated in this section using the modified 
Bessel function $K_{\alpha}(x)$. The classical Bessel function 
$J_{\alpha}(x)$ is usually defined via the differential equation
\begin{equation}
\frac{d^{2}y}{dx^{2}} + \frac{1}{x} \frac{dy}{dx} - 
\left( 1 + \frac{\alpha^{2}}{x^{2}} \right) y = 0.
\label{diff-bes}
\end{equation}
\noindent
The modified function $I_{\alpha}(x) = e^{-\alpha \pi i/2} J_{\alpha}
(xe^{\pi i/2})$ is also real valued. In the case $\alpha \not \in 
\mathbb{Z}$, the functions $I_{\alpha}(x)$ and $I_{-\alpha}(x)$ form a basis
for the solutions of (\ref{diff-bes}). To deal with the situation of 
integer parameters, it is convenient to define
\begin{equation}
K_{\alpha}(x) = \frac{\pi}{2 \, \sin \pi \alpha} 
\left[ I_{-\alpha}(x) - I_{\alpha}(x) \right].
\end{equation}
\noindent
This is the modified Bessel function of second kind, also called 
MacDonald function. 

The proof of (\ref{identity-1}) proceeds along the following lines: the first 
step is to obtain an expression for $N_{0,4}(a,m)$ as an integral involving 
$K_{1/4}(x)$. An entry in \cite{prudnikov1}  gives a hypergeometric form of 
this integral. The final step is to prove a hypergeometric identity that 
transforms this form to a result in \cite{bomohyper}. Details about Bessel 
functions can be found in \cite{andrews3} and chapter 10 of \cite{nist}. 

\smallskip

The identity 
\begin{equation}
\int_{0}^{\infty} e^{-c u} u^{m} \, du  = 
\frac{\Gamma(m+1)}{c^{m+1}},
\end{equation}
\noindent
valid for $\realpart{c}>0$, is used with $c = x^{4}+2ax^{2}+1$ to produce 
\begin{equation}
\frac{1}{(x^{4} + 2ax^{2} + 1)^{m+1}} = 
\frac{1}{\Gamma(m+1)} \int_{0}^{\infty} e^{-u} u^{m} 
e^{-u(x^{4}+2ax^{2})} \, du.
\end{equation}
\noindent
Integration yields 
\begin{equation}
N_{0,4}(a;m) = \frac{1}{\Gamma(m+1)} \int_{0}^{\infty} e^{-u} 
u^{m} \int_{0}^{\infty} e^{-u(x^{4}+2ax^{2})} \, du \, dx. 
\end{equation}
\noindent
The restriction $\realpart{c}>0$ is satisfied by taking $-1 <  a < 1$.

Entry $3.469.1$ in \cite{gr} gives 
\begin{equation}
\int_{0}^{\infty} e^{-\mu x^{4} - 2 \nu x^{2}} \, dx = 
\frac{1}{4} \sqrt{\frac{2 \nu}{\mu}} \text{exp} \left( \frac{\nu^{2}}
{2 \mu} \right) K_{\small{1/4}} \left( \frac{\nu^{2}}{2 \mu} \right).
\end{equation}

The choice
$\mu = u$ and $\nu = a u$ yields:

\begin{Lem}
\label{lemma-bes1}
The quartic integral $N_{0,4}(a;m)$ is given by
\begin{equation}
N_{0,4}(a;m) = \frac{2^{m - \tfrac{1}{2}}}{\Gamma(m+1) \, a^{2m+ \tfrac{3}{2}}} 
\int_{0}^{\infty} t^{m} e^{-bt} K_{\small{1/4}}(t) \, dt,
\end{equation}
\noindent
where $b = 2/a^{2}-1$. 
\end{Lem}

This expression can be evaluated using \cite[vol.2, 2.16.6.2]{prudnikov1}
\begin{equation}
\int_{0}^{\infty}x^{\alpha-1}e^{-px}K_{\nu}\left(cx\right)dx
=\frac{\left(2c\right)^{\nu}\sqrt{\pi}}{\left(p+c\right)^{\alpha+\nu}}
\frac{\Gamma(\alpha-\nu) \, \Gamma(\alpha+\nu)}{\Gamma(\alpha + \tfrac{1}{2})}
\,_{2}F_{1}\left(\begin{array}{c}
\alpha+\nu,\nu+\frac{1}{2}\\
\alpha+\frac{1}{2}\end{array}\Big{\vert}\frac{p-c}{p+c}\right)
\nonumber
\end{equation}
valid for $\realpart{(c+p)}>0,\,\,\realpart{\alpha} > 
\vert \realpart{\nu} \vert.$  This yields the expression
\begin{equation}
N_{0,4}(a;m) = 
\frac{\sqrt{\pi} a}{2 \sqrt{2} m!} 
\frac{\Gamma\left(m+\tfrac{3}{4}\right) \Gamma(m + \tfrac{5}{4})}
{\Gamma\left(m+ \tfrac{3}{2} \right)}
\,_{2}F_{1}\left(\begin{array}{c}
m+\frac{5}{4},\frac{3}{4}\\
m+\frac{3}{2}\end{array} \Big{\vert} \, 1-a^{2}\right).
\end{equation}

Using \cite[15.3.24]{abramowitz1} this can be written as

\begin{equation}
N_{0,4}\left(a;m\right) 
=\frac{\sqrt{\pi}}{2}\frac{\Gamma\left(2m+\frac{3}{2}\right)}
{\Gamma\left(2m+2\right)}\,_{2}F_{1}\left(\begin{array}{c}
2m+\frac{3}{2},\frac{1}{2}\\
m+\frac{3}{2}\end{array}\Big{\vert}\frac{1-a}{2}\right).
\label{original}
\end{equation}

\smallskip

The proof of (\ref{identity-1}) is now based on the representation
\begin{equation}
N_{0,4}\left(a;m\right) 
= \frac{2^{m-\tfrac{1}{2}}}{(a+1)^{m+1/2}} B\left( 2m + \tfrac{3}{2}, 
\tfrac{1}{2} \right) 
\,_{2}F_{1}\left(\begin{array}{c}
-m,m+1\\
m+\frac{3}{2}\end{array}\Big{\vert}\frac{1-a}{2}\right)
\end{equation}
\noindent
described in \cite{bomohyper}. In particular, the polynomial $P_{m}(a)$
is identified there as the Jacobi polynomial $P_{m}^{(m+1/2,-m-1/2)}(a)$. 

Matching  both representations for $N_{0,4}(a;m)$ shows that (\ref{identity-1})
follows from the next result.

\begin{Prop}
\label{hyper-iden1}
The identity 
\begin{equation}
_{2}F_{1}\left(\begin{array}{c}
-m,m+1\\
m+\frac{3}{2}\end{array}\Big{\vert} z \right) = 
(1-2z) (1-z)^{m+\tfrac{1}{2}} \, \,  
_{2}F_{1}\left(\begin{array}{c}
m + \tfrac{5}{4}, \tfrac{3}{4}\\
m+\frac{3}{2}\end{array}\Big{\vert} \, 4z(1-z) \right)
\label{hyp-0}
\end{equation}
\noindent
holds.
\end{Prop}
\begin{proof}
The WZ-method shows that both sides satisfy the recurrence
\begin{multline}
(4m+7)(4m+9)za_{m+2} + (2m+3)(2m+5)(4z^{2}-4z-1)a_{m+1} \label{recur-11} \\ 
-(2m+3)(2m+5)(z-1)a_{m} = 0.
\end{multline}

To complete the proof, it suffices to check two initial values. The required 
identities for $m=0$ and $m=1$ are written in terms of 
$t = 4z(1-z)$ in the form
\begin{equation}
\label{init-1}
_{2}F_{1}\left(\begin{array}{c}
\tfrac{5}{4}, \tfrac{3}{4}\\
\frac{3}{2}\end{array}\Big{\vert} \, t \right)
= \frac{\sqrt{2}}{(1-t)^{1/2} \, \sqrt{1 + \sqrt{1-t}}}
\end{equation}
\noindent
and 
\begin{equation}
\label{init-2}
_{2}F_{1}\left(\begin{array}{c}
\tfrac{9}{4}, \tfrac{3}{4}\\
\frac{5}{2}\end{array}\Big{\vert} \, t \right)
= \frac{2 \sqrt{2} (3 + 2 \sqrt{1-t})}{5(1+\sqrt{1-t})^{3/2} \, \sqrt{1-t}}.
\end{equation}

The proofs of these identities are elementary using the fact that the 
hypergeometric differential equation
\begin{equation}
t(1-t) y'' - \left[ c - (a+b+1)t \right]y' - aby = 0. 
\end{equation}
\noindent
has a unique solution analytic at $t=0$ with initial value $y(0)=1$. Indeed, 
both sides of (\ref{init-1}) satisfy 
\begin{equation}
t(1-t)y'' + \left( \tfrac{3}{2} - 3t \right)y' - \tfrac{15}{16}y = 0
\end{equation}
\noindent
with value $y(0)=1$. This establishes (\ref{init-1}). The same holds for 
(\ref{init-2}) using 
\begin{equation}
t(1-t)y'' + \left( \tfrac{5}{2} - 4t \right)y' - \tfrac{27}{16}y = 0
\end{equation}
\end{proof}

\begin{Note}
The identity in Proposition \ref{hyper-iden1} can also be established using 
C. Koutschan package \texttt{HolonomicFunctions} \cite{Koutschan10b}. Denote 
the left hand side of (\ref{hyp-0}) by $F1$ and the right hand side by $F2$. 
The command 
\begin{center}
\texttt{Annihilator}$[F1,S[m]]$
\end{center}
\noindent
 finds the operator
\begin{equation}
(7+4m)(9+4m)zS_{m}^{2} + 
(3 +2m)(5+2m)(-1-4z+4z^{2})S_{m} - 
(3+2m)(5+2m)(-1+z)
\nonumber
\end{equation}
\noindent
with $S_{m}$ being the shift in the discrete parameter $m$; that is, 
$S_{m}g(m;x) = g(m+1;x)$. The 
package claims that this 
operator annihilates the left hand side of (\ref{hyp-0}). This is
precisely the recurrence (\ref{recur-11}) obtained before. The command 
\begin{center}
\texttt{Annihilator}$[F1,S[m]]$ == 
\texttt{Annihilator}$[F2,S[m]]$
\end{center}
\noindent
returns \texttt{True}, showing that the right 
hand side $F2$ satisfies the same recurrence. The initial conditions can also
be determined automatically. 
\end{Note}

\section{A second proof of the main identity via  Bessel functions} \label{S:Bessel1} 

The previous section contains the identity (\ref{lemma-bes1}) giving the 
quartic integral $N_{0,4}(a,m)$ as an integral involving Bessel functions:
\begin{equation}
N_{0,4}(a;m) = \frac{2^{m - \tfrac{1}{2}}}{\Gamma(m+1) \, a^{2m+ \tfrac{3}{2}}} 
\int_{0}^{\infty} t^{m} e^{-bt} K_{\small{1/4}}(t) \, dt.
\label{form-00}
\end{equation}
\noindent
The table of integrals \cite{gr} contains, as Entry $6.611.3$, the 
special case $m=0$:
\begin{equation}
\int_{0}^{\infty} e^{-bt} K_{\small{1/4}}(t) \, dt = 
\frac{\pi}{\sqrt{2} \, \sqrt{b^{2}-1}} 
\left[ (b + \sqrt{b^{2}-1})^{1/4} - (b+\sqrt{b^{2}-1})^{-1/4} \right]
\end{equation}
\noindent
that can be written as 
\begin{equation}
\int_{0}^{\infty} e^{-bt} K_{\small{1/4}}(t) \, dt = 
\frac{\pi}{2} \frac{a^{3/2}}{\sqrt{1+a}},
\label{bes-11}
\end{equation}
\noindent 
using $b = a^{2}/2-1$. Differentiating 
$m$ times with respect to $b$, it follows that 
\begin{equation}
N_{0,4}(a,m) = \frac{(-1)^{m} 2^{m-1/2}}{m! \, a^{2m+3/2}}
\frac{\partial^{m}}{\partial b^{m}} \left( \frac{\pi}{2} 
\frac{a^{3/2}}{\sqrt{1+a}}\right).
\label{bes-12}
\end{equation}

The proof of (\ref{identity-1}) is thus reduced to finding an analytic 
expression for the derivatives in (\ref{bes-12}). 

\begin{Prop}
There exists a polynomial $Q_{m}(a)$ such that 
\begin{equation}
\frac{\partial^{m}}{\partial b^{m}} 
\left(\frac{a^{3/2}}{\sqrt{1+a}} \right) 
= \frac{(-1)^{m} m!}{2^{2m}} \frac{a^{2m+3/2}}
{(1+a)^{m+1/2}} Q_{m}(a).
\end{equation}
\end{Prop}
\begin{proof}
The case $m=0$ holds with $Q_{0}(m) = 1$.  To complete 
the inductive step, it is shown that if $Q_{m}(a)$ is a polynomial then the 
function $Q_{m+1}(a)$ defined by the relation 
\begin{equation}
\frac{\partial}{\partial b} 
\left( 
\frac{(-1)^{m} m!}{2^{2m}} \frac{a^{2m+3/2}}
{(1+a)^{m+1/2}} Q_{m}(a) \right) = 
\frac{(-1)^{m+1} (m+1)!}{2^{2m+2}} \frac{a^{2m+7/2}}
{(1+a)^{m+3/2}} Q_{m+1}(a),
\label{def-tn}
\end{equation}
\noindent
is also a polynomial. 

The chain rule and
$\frac{\partial a}{\partial b} 
= - \tfrac{1}{4}a^{3}$ show that (\ref{def-tn}) is equivalent to 
\begin{multline}
2(m+1)Q_{m+1}(a) = 
\left[2(m+1)(a+1) + (2m+1) \right]Q_{m}(a) + \label{rec-Tm} \\
\left[ 2(a+1)^{2} - 2(a+1) \right]Q_{m}'(a).
\end{multline}
\noindent
Thus $Q_{m+1}(a)$ is a polynomial in $a$.
\end{proof}

The proof of (\ref{identity-1}) is now reduced to checking that the polynomial
$Q_{m}$ in the previous lemma is given by $P_{m}(a)$ defined 
in (\ref{polyP-def}). 

\begin{Thm}
The polynomial $Q_{m}(a)$ is the same as $P_{m}(a)$.
\end{Thm}
\begin{proof}
It suffices to check that 
$P_{m}(a)$ satisfies the same recurrence as $Q_{m}(a)$. To this end, compare 
the coefficients of $(1+a)^{k}$ on both sides of (\ref{rec-Tm}). The result 
is equivalent to 
\begin{eqnarray} 
2^{k} (m+1) \binom{2m-2k+2}{m-k+1} \binom{m+k+1}{m+1} & = & 
2^{k+1} (m+1)\binom{2m-2k+2}{m-k+1} \binom{m+k-1}{m}  \nonumber \\
& + &  2^{k+1} (2m+1)\binom{2m-2k}{m-k} \binom{m+k}{m}  \nonumber \\
& + &  2^{k+1} (k-1)\binom{2m-2k+2}{m-k+1} \binom{m+k-1}{m}  \nonumber \\
& - &  2^{k+2} k\binom{2m-2k}{m-k} \binom{m+k}{m}. \nonumber 
\end{eqnarray}

\noindent
This however can routinely be verified. Simply 
divide through by $\binom{2m-2k}{m-k} \binom{m+k}{m}$
and the statement reduces to a simple polynomial identity. 
\end{proof}

\medskip

A small variation of this  proof of (\ref{identity-1}) is obtained by 
differentiating 
\begin{equation}
\int_{0}^{\infty} e^{-bt} K_{\tfrac{1}{4}}(t) \, dt = 
\frac{\pi}{2^{1/4}(b+1)^{1/2} \, \sqrt{\sqrt{2} + \sqrt{b+1}}}.
\end{equation}
directly with respect to the parameter $b$. The first few 
examples suggest the next result.

\begin{Lem}
There are polynomials $S_{m}, \, T_{m}$ such that
\begin{multline}
\frac{d^{m}}{db^{m}} 
\left( \frac{1}{(b+1)^{1/2} \, \sqrt{\sqrt{2} + \sqrt{b+1}}} \right) = 
\label{formula-77} \\
(-1)^{m} \frac{S_{m}(b) + \sqrt{b+1} T_{m}(b)}
{2^{2m} (b+1)^{m+1/2} \, (\sqrt{2} + \sqrt{b+1})^{m+1/2}}.
\end{multline}
\end{Lem}
\begin{proof}
The proof is by induction on $m$, and is obtained upon differentiating the 
stated expression for the $m$-th derivative. The base case $m=0$ is obvious.
Assume (\ref{formula-77}) holds for $m$. Then differentiating the right
hand side of (\ref{formula-77}) generates the recurrence
\begin{eqnarray}
S_{m+1}(b) & = & 2 \sqrt{2}(2m+1) S_{m}(b) \label{recu-ST} \\
 & & -(b+1) \left[ 4 \sqrt{2} S_{m}'(b) - 
(1+6m) T_{m}(b) + 4(b+1)T_{m}'(b) \right] \nonumber \\
T_{m+1}(b) & = & 3(2m+1)S_{m}(b) + 4 \sqrt{2}m T_{m}(b) \nonumber \\
 & & - 4(b+1) \left( S_{m}'(b) + \sqrt{2} T_{m}'(b) \right), \nonumber 
\end{eqnarray}
\noindent
where $b = 2/a^{2}-1$. The initial condition $S_{0}(b)=1$ and $T_{0}(b)=0$ 
yield the result.
\end{proof}

\begin{Lem}
Let $m' = \lfloor{ m/2 \rfloor}$. Define 
\begin{equation}
U_{m}(b) = \frac{2^{-2m}}{(1+b)^{m'}} 
\sum_{k=0}^{m} 2^{k} \binom{2m-2k}{m-k} \binom{m+k}{m} 
\sum_{j=0}^{\lfloor{k/2 \rfloor}} \binom{k}{2j} 2^{j} (1+b)^{m'-j} 
\nonumber
\end{equation}
\noindent
and 
\begin{equation}
V_{m}(b) = \frac{2^{-2m}}{(1+b)^{m'}} 
\sum_{k=0}^{m} 2^{k} \binom{2m-2k}{m-k} \binom{m+k}{m} 
\sum_{j=0}^{\lfloor{(k-1)/2 \rfloor}} \binom{k}{2j+1} 2^{j} (1+b)^{m'-j}.
\nonumber
\end{equation}
\noindent
Then 
\begin{equation}
P_{m}(a) = U_{m}(b) + a V_{m}(b).
\end{equation}
\end{Lem}
\begin{proof}
The polynomial $P_{m}(a)$ 
is decomposed into its even and odd part using 
\begin{equation}
(1+a)^{k} = \frac{1}{2} \left[ (1+a)^{k} + (1-a)^{k} \right] +
            \frac{1}{2} \left[ (1+a)^{k} - (1-a)^{k} \right].
\end{equation}
\noindent
The result follows.
\end{proof}

\smallskip

The proof of (\ref{identity-1}) now reduces to expressing the polynomials 
$S_{m}(b)$ and $T_{m}(b)$ in terms of $U_{m}(b)$ and $V_{m}(b)$. This is 
given by 
\begin{equation}
T_{m}(b) = \frac{m! 2^{3m/2}}{\sqrt{2} a^{m-1}} 
\times \begin{cases}
U_{m}(b) & \quad \text{ if } m \text{ is odd} \\
aV_{m}(b) & \quad \text{ if } m \text{ is even}
\end{cases}
\end{equation}
\noindent
and
\begin{equation}
S_{m}(b) = \frac{m! 2^{3m/2}}{\sqrt{2} a^{m}} 
\times \begin{cases}
aV_{m}(b) & \quad \text{ if } m \text{ is odd} \\
U_{m}(b) & \quad \text{ if } m \text{ is even}.
\end{cases}
\end{equation}
\noindent
The details are elementary and are left to the reader.

\begin{Note}
The identity (\ref{bes-12}) yields
\begin{equation}
N_{0,4}(a,m) = \frac{\pi \, a^{m-1}}{m! 2^{5m/2+3/2} \, (a+1)^{m+1/2}}
\left[ \sqrt{2} T_{m}(b) + S_{m}(b) \right].
\label{form-99}
\end{equation}
\end{Note}

\smallskip
\medskip

\section{Conclusions} \label{S:conc} 

The value of the definite integral
\begin{equation}
N_{0,4}(a;m) = \int_{0}^{\infty} \frac{dx}{(x^{4}+2ax^{2}+1)^{m+1}},
\nonumber
\end{equation}
\noindent 
is given by 
\begin{equation}
N_{0,4}(a;m)  =   \frac{\pi}{2^{m+3/2} (a+1)^{m+1/2} } 
P_{m}(a)  
\nonumber
\end{equation}
\noindent 
where 
\begin{equation}
P_{m}(a) = 2^{-2m} \sum_{k=0}^{m} 2^{k} \binom{2m-2k}{m-k} \binom{m+k}{k} 
(a+1)^{k}.
\nonumber
\end{equation}

\medskip

Several proofs of this identity appeared in the literature. The current 
work includes proofs relating this identity to Bessel functions, Schur 
polynomials and a method of Schwinger for the evaluation of definite 
integrals.

\medskip

\no
{\bf Acknowledgements}. The authors wish to thank Armin Straub for 
comments on C. Koutschan package. The work of the second author was 
partially supported by NSF-DMS 0070567. C. Vignat thanks V. Moll for his
invitation to visit Tulane in July 2010. \\

\end{document}